\documentclass[11pt, english]{article}
 \pdfoutput=1
\usepackage{hyperref}
\usepackage[utf8]{inputenc}
\usepackage{amsxtra}
\usepackage{color}
\usepackage{amsopn}
\usepackage{amsmath,amsthm,amssymb,amscd}
\usepackage{mathtools}
\usepackage{eufrak}
\usepackage{tikz}
\usepackage{tikz-cd}
\usepackage{calrsfs}
\usepackage[all,cmtip]{xy}
\usepackage{footmisc}

\theoremstyle{remark}
\newtheorem{remark}{Remark}[section]

\theoremstyle{definition}
\newtheorem{defn}[remark]{Definition}
\newtheorem{ex}[remark]{Example}
\newtheorem{rem}[remark]{Remark}

\theoremstyle{theorem}
\newtheorem{thm}[remark]{Theorem}

\newtheorem{corol}[remark]{Corollary}
\newtheorem{lemma}[remark]{Lemma}

\newcommand{\lie}[1]{\mathfrak{#1}}

\newcommand{\G}{G_2}

\newcommand{\beq}{\begin{equation}}
\newcommand{\eeq}{\end{equation}}
\newcommand{\bqn}{\begin{eqnarray}}
\newcommand{\eqn}{\end{eqnarray}}
\newcommand{\bqne}{\begin{eqnarray*}}
\newcommand{\eqne}{\end{eqnarray*}}

\newcommand{\R}{{\mathbb R}}

\newcommand{\C}{{\mathbb C}}

\newcommand{\SU}{{\rm SU}}
\newcommand{\SO}{{\rm SO}}

\newcommand{\GL}{\mathrm{GL}}

\newcommand{\hg}{\mathfrak{h}}

\newcommand{\SL}{\mathrm{SL}}

\usepackage[affil-it]{authblk}

\begin{document} 
\title{The Laplacian coflow on almost-abelian Lie groups}

\author[1]{Leonardo Bagaglini}
\author[2]{Anna Fino} 
\affil[1]{
Dipartimento di Matematica e Informatica \lq \lq Ulisse Dini\rq \rq, Universit\`a degli Studi di Firenze, Viale Giovan Battista Morgagni, 67/A , 50134 Firenze, Italy\\\par
E-mail address: \emph{leonardo.bagaglini@unifi.it}}
\affil[2]{
Dipartimento di Matematica \lq\lq Giuseppe Peano\rq\rq \\ Universit\`a degli Studi di Torino,
Via Carlo Alberto 10,
10123 Torino, Italy\\\par
E-mail address: \emph{annamaria.fino@unito.it}
 }

\maketitle

\begin{abstract}
We find explicit solutions of the Laplacian coflow of $\G-$structures on seven-dimensional almost-abelian Lie groups.
Moreover,  we construct  new   examples of solitons for the Laplacian coflow which are not eigenforms of the Laplacian and we exhibit a solution, which is not a soliton, having a bounded interval of existence. 
\end{abstract}
\section{Introduction}
A  $G_2$-structure on a  seven-dimensional manifold  $M$ is given by a $3$-form $\varphi$ on $M$ with pointwise stabilizer isomorphic to the exceptional  group   $G_2 \subset  SO(7)$. The $3$-form $\varphi$ induces  a Riemannian  metric  $g_{\varphi}$, an orientation and so a Hodge star operator $\star_{\varphi}$  on $M$. It is well-known \cite{FernandezGray} that  $\varphi$  is parallel with respect to the Levi-Civita connection of $g_{\varphi}$  if
and only if $\varphi$  is closed and coclosed and that when this happens  the holonomy of $g_{\varphi}$  is contained
in  $G_2$.

The different classes of $G_2$-structures can be described in terms of the exterior
derivatives $d \varphi$  and $d \star_{\varphi} \varphi$ \cite{Bryant2,FernandezGray}.  If $d \varphi$= 0, then the $G_2$-structure is called closed (or
calibrated in the sense of Harvey and Lawson \cite{HarveyLawson}) and if $\varphi$  is coclosed, then the $G_2$-structure is called coclosed (or cocalibrated \cite{HarveyLawson}).

Flows of $G_2$-structures were  first considered by  Bryant in \cite{Bryant2}. In particular, he  considered the
Laplacian  flow of closed  $G_2$-structures.  Recently, Lotay
and Wei investigated the properties of the Laplacian 
flow in the series of papers \cite{LW1, LW2, LW3}.  The Laplacian  coflow has been originally
proposed by Karigiannis, McKay and Tsui in \cite{KMT}  and, for an initial coclosed $G_2$-form
$\varphi_0$  with   $\star_{\varphi_0}   \varphi_0 = \phi_0$, it is given by
\begin{equation}
\label{coflowgeneral}
\frac{\partial}{\partial t} \phi(t) = - \Delta_t \phi(t), \quad d \phi (t) =0, \quad \phi(0) = \phi_0,
\end{equation}
where   $\phi(t)$ is the Hodge dual 4-form of a $G_2$-structure $\varphi(t)$ with respect to the Remannian metric $g_{\varphi(t)}.$ This
flow preserves the condition of the $G_2$-structure being coclosed and it was studied in \cite{KMT} for  warped products of an interval, or a circle, with a compact
6-manifold $N$ which is taken to be either a nearly  K\"ahler  manifold or a Calabi-Yau
manifold.  No general result is known about the short time existence of  the coflow \eqref{coflowgeneral}.  In \cite{BFF2} the Laplacian coflow   on the seven-dimensional Heiseberg group has  been studied, showing that  the solution is always ancient, that is it is defined in some interval
$(- \infty, T)$, with $0 < T < + \infty.$
Other examples of flows of $G_2$-structures are the modified Laplacian coflow \cite{Gri13, Gri15} and Weiss and Witt's heat flow \cite{WW12}. The first one  is  a flow of coclosed $G_2$-structures obtained by adding a fixing term to the Laplacian coflow in order to ensure weak parabolicity in the exact directions. The second one  is the gradient flow associated to the functional which measures the full torsion tensor of a $G_2$-structure; generally it does not preserve any special class of $G_2$-structures  but it can be modified to fix the underlying metric (see \cite{Bag17}).\par
As for the Ricci flow (and other geometric  flows),  for the Laplacian coflow it is  interesting to consider  {\em self-similar solutions}  which are evolving by diffeomorphisms and scalings.  If $x_t$  is a 1-parameter family of diffeomorphisms
generated by a vector field $X$ on $M$ with $x_0=\mathrm{Id}_M$ and  $c_t$ is a positive real function on $M$ with $c_0=1$, then  a coclosed $G_2$-structure  $\phi(t) = c_t (x_t)^* \phi_0$ is a solution of the coflow \eqref{coflowgeneral} if and only if $\phi_0$ satisfies
$$
- \Delta_0 \phi_0   = L_X \phi_0   + c'_0   \phi_0 = d(X  \neg  \phi_0) + c'_0 \phi_0,
$$
where  by $L_X$ and $ X\neg$  we  denote respectively the Lie derivative and the contraction with the vector field $X$.  A coclosed $G_2$-structure satisfying the  previous equation is called {\em soliton}.  As in the case of  the Ricci flow,  the soliton is  said to be expanding, steady, or shrinking if  $c'_0$ is positive, zero, or negative, respectively. By Proposition 4.3 in \cite{KMT}, if  $M$  is compact, then there are no expanding or steady soliton solutions of \eqref{coflowgeneral},   other than the trivial case of a torsion-free $G_2$-structure in the steady case.  Examples  of  solitons  for the Laplacian flow  have been constructed in   \cite{FFM, Lauret, Lauret2, Lauret3, Nicolini, FR}.

In this paper, we study the coflow  \eqref{coflowgeneral} on almost abelian Lie groups, i.e.,  on  solvable Lie groups with a codimension-one abelian normal subgroup.
Coclosed
and closed  left-invariant $G_2$-structures on almost-abelian  Lie groups have been studied by Freibert in  \cite{Freibert1, Freibert2}. General obstructions to the existence of a coclosed $G_2$-structure on a Lie
algebra  of dimension seven with non-trivial center have been given in \cite{BFF}.

By \cite{Lauret} the Laplacian coflow  on homogeneous spaces can be completely described as a  flow of Lie brackets on the ordinary euclidean space, the so-called \emph{bracket flow}.   In particular,  Lauret showed in \cite{Lauret} that any left-invariant closed Laplacian flow solution $\varphi (t)$  on an almost abelian  Lie group is immortal, i.e.,  defined in the  interval $[0, + \infty)$.  Moreover,  he proved that the scalar curvature of  $g_{\varphi(t)}$ is strictly increasing and
converges to zero as  $t$ goes to $+ \infty$.

In Section \ref{SectionExplicitsol} we find an explicit description of the left-invariant  solutions  to the Laplacian coflow  on almost-abelian Lie groups     under suitable assumptions on the initial data,  showing  that  the solution 
is  ancient.

In Section \ref{SectionSoliton}   we show  sufficient  conditions for a left-invariant  coclosed $G_2$-structure on an almost-abelian Lie  group to be a soliton for the Laplacian  coflow. In particular we construct new examples of solitons which are not eigenforms of the  Laplacian.

\section{Preliminaries}
A $k$-form on  an $n$-dimensional  real vector space  is stable if it lies in an open orbit of the linear group $GL(n, \R)$.
In this section we review the theory of stable forms in  dimensions six and seven. We refer to \cite{CLSS,Hitchin}, and the references therein, for more details.
Throughout the sections we denote by $\vartheta$ and by $^*$ the actions of the endomorphism group and the general linear group respectively.
\subsection{Linear $\G-$structures}
A $3$-form $\varphi$ on a seven-dimensional real vector space $V$ is \emph{stable} 
if   the $\Lambda^7(V^*)$-valued bilinear form $b_\varphi$, defined by 
$$b_\varphi(x,y)=\frac{1}{6}(x\neg \varphi)\wedge(y\neg\varphi)\wedge\varphi,\quad x,y\in V,$$
is nondegenerate. In this case $\varphi$ defines an orientation $vol_\varphi$ by $\sqrt[9]{\mathrm{det}b_\varphi}$ and
a bilinear form $g_\varphi$ by $b_\varphi=g_\varphi vol_\varphi$.
A stable $3$-form $\varphi$ is said to be \emph{positive}, and we will write $\varphi\in\Lambda^3_+V^*$, if, in addition, $g_\varphi$ is positive definite.\par
It is a well-known  fact that the action of $\mathrm{GL}(V)$ on $\Lambda^3_+V^*$ is transitive and the stabiliser of  every  $\varphi \in \Lambda^3_+V^*$ is a subgroup of $\SO(g_\varphi)$ isomorphic to $\G$. 
Therefore,  if we assume that  $||\varphi||_{g_\varphi} =7$  we get a one-to-one correspondence between
normalized positive $3$-forms  on $V$ and presentations of $\G$ inside $\GL(V)$.\par 
We denote by $\star_\varphi$ the Hodge operator induced by $\varphi$ and we will always write $\phi$ to indicate the Hodge dual form $\star_\varphi\varphi$ of $\varphi$. Precisely
$\phi$ belongs to the $\mathrm{GL}(V)-$orbit, denoted by $\Lambda^4_+V^*$, of \emph{positive} $4$-forms. 
It is another basic fact that the stabilisers of $\varphi$ and $\phi$ under $\mathrm{GL}^+(V)$ are equal and therefore the choice of $\phi$ and  of an orientation $vol$ 
is sufficient  to define $\varphi$. \par
We will refer to a presentation of $\G$ inside $\mathrm{GL}(V)$ as a \emph{linear $\G-$structure} on $V$, and we will call $\varphi$ and $\phi$ the \emph{fundamental} forms associated to the linear $\G-$structure.\par
On $V$ there  exists always a $g_\varphi$-orthonormal and positive oriented co-frame $\left(e^1,\dots,e^7\right)$, called an \emph{adapted} frame, such that 
\begin{align*}
\varphi&=e^{127}+e^{347}+e^{567}+e^{135}-e^{146}-e^{236}-e^{245},\\
\phi&=e^{1234}+e^{3456}+e^{1256}-e^{2467}+e^{1367}+e^{1457}+e^{2357}.
\end{align*}
\subsection{Linear $\SU(3)-$structures}
Let $U$ be a  real vector space of dimension six. A $2$-form $\omega$ on $U$ is  \emph{stable} if  it  is nondegenerate, i.e.,  if 
$\omega^3 \neq 0.$

Given a $3$-form $\psi$ on $U$, the equivariant identification of $\Lambda^5U^*$ with $U\otimes\Lambda^6U^*$ allows us to define the operator 
$$K_\psi:U\rightarrow U\otimes\Lambda^6U^*,\quad  x \mapsto (x\neg\psi)\wedge\psi.$$
We can consider  the trace of its second iterate 
$$\lambda(\psi)=\frac{1}{6}\mathrm{tr}(K^2_\psi)\in(\Lambda^6U^*)\otimes (\Lambda^6U^*),$$
where
$$
 K^2_\psi:U\rightarrow U\otimes(\Lambda^6U^*)\otimes (\Lambda^6U^*).
 $$
Then   $\psi$  is  \emph{stable} if  and only if $\lambda(\psi)\neq 0$. If  $\lambda(\psi)<0$, the $3$-form $\psi$   is called  \emph{negative}.  In this case we will write $\psi\in\Lambda^3_-U$.
Here the basic fact is that the action of $\GL^+(U)$ is transitive on $\Lambda^3_-U$ with stabiliser of $\psi$ isomorphic to $\SL(3,\C)$, where the associated complex structure $J_\psi$ and complex volume form $\Psi$ on $U$ are given  respectively by
$$J_\psi=\frac{1}{\sqrt{-\lambda}}\,K_\psi,\quad\Psi=-J^*_\psi\psi+i\psi.$$
It is important to note that the  $3$-form $J_\psi^*\psi$ is still negative and that it defines the same complex structure of $\psi$.\par
If $\psi$ is a negative $3$-form and $\omega$ a stable $2$-form,  then $\omega$ is of type $(1,1)$ with respect to $J_\psi$, meaning that $J_\psi^*\omega=\omega$, if and only if  $\psi\wedge\omega=0$.
In this case we can define a symmetric bilinear form    $h$ on $U$ by
$$h(x,y)=\omega(x,J_\psi y),\quad x,y\in U.$$
When  $h$ is positive definite,  the couple $(\omega,\psi)$ is  said to be a \emph{positive couple} and it  defines a \emph{linear $\SU(3)-$structure}, meaning that its stabiliser in $\GL(U)$ is isomorphic to
$\SU(3)$. In this case  $h$ is hermitian with respect to $J_\psi$ and $\Psi=-J^*_\psi\psi+i\psi$ is a complex volume form. 
A positive couple is said to be \emph{normalized} if 
$$2\,\omega^3=3\,\psi\wedge J^*_\psi\psi.$$
If  a positive couple $(\omega,\psi)$ is normalized, then there exists an $h-$orthonormal and positive oriented co-frame of $U$, called  an \emph{adapted} frame,  $\left(f^1,J^*f^1,f^2,J^*f^2,f^3,J^*f^3\right)$ such that
\begin{align*}
\omega&=f^1\wedge J^*f^1+f^2\wedge  J^*f^2+f^3 \wedge J^*f^3,\\
\psi&=-f^2\wedge f^4\wedge f^6+f^1\wedge J^*f^3\wedge J^*f^6+J^*f^1\wedge f^4\wedge J^*f^5+J^*f^2\wedge J^*f^3\wedge f^5.
\end{align*}
Therefore, if we denote by $*_h$ the Hodge operator on $U$ associated to $h$, it follows that
$$*_h \omega=\frac{1}{2}\omega^2,\quad *_h \psi=J^*_\psi\psi.$$
\subsection{From $\G$ to $\SU(3)$}
Given a linear $\G-$structure $\varphi$ on $V$, with fundamental forms $\varphi$ and $\phi$, the six-dimensional sphere $$S^6 =\{ x \in V \mid g_{\varphi} (x, x) = 1 \} \subset V$$ is $\G-$homogeneous and, for any non-zero vector $v\in S^6$, there is an induced linear $\SU(3)-$structure on  the $g_{\varphi}$-orthogonal complement   $U = ({\mbox {span}} <v>)^{\perp}$.  This structure is constructed  as follows. Let 
$$\omega=v\neg\varphi,\quad \psi=-v\neg \phi.$$
Then $(\omega,\psi)$ is a positive couple on $U$ defining the  linear $\SU(3)-$structure. It is then clear that the restriction of an adapted co-frame of $(V, \varphi)$, with $v=e_7$, to $U$ gives an adapted frame of $(U, \omega, \psi)$ and it follows that
$$\varphi=\omega\wedge e^7-J^*_\psi\psi,\quad \phi=\frac{1}{2}\omega^2+ \psi \wedge e^7.$$
\section{Explicit solutions  to the Laplacian coflow on almost-abelian Lie groups} \label{SectionExplicitsol}

We recall that a Lie group $G$  is said to be  {\em almost-abelian}  if its Lie algebra  $\frak g$  has a codimension one abelian ideal $\frak h$. Such a Lie algebra will be called
almost-abelian, and it can be written as  a semidirect product $\frak g = \R x  \ltimes_A \frak h$.   We point out that an almost-abelian
Lie algebra is nilpotent if and only if the operator  $ad_x \vert_{\frak h}$  is nilpotent.

Freibert showed in \cite{Freibert1} that  if $\frak g$  is  a 7-dimensional  almost-abelian Lie algebra, then, the following are equivalent:
\begin{enumerate}
\item  $\frak g$  admits a coclosed  $G_2$-structure $\varphi$.

\item For any  $x  \in \frak g \setminus \frak h$, $ad(x) \vert_\frak{h}   \in  \frak{gl}(\frak h)$  belongs to $ \frak{sp}( {\frak h},  \omega),$ where $\omega$ is a   non-degenerate
$2$-form  $\omega$ on  $\frak h.$

\item   For any  $x \in \frak g \setminus   \frak h$, the complex Jordan normal form of $ad(x) \vert_\frak{h}$  has the property that for
all $m \in  \mathbb N$ and all  $\lambda \neq 0$  the number of Jordan blocks of size $m$  with $\lambda$ on the diagonal is the same as the number of Jordan blocks of size $m$ with  $- \lambda$ on the diagonal and the number of Jordan blocks of size $2m - 1$  with $0$  on the diagonal is even.
\end{enumerate}

\smallskip

In this section we obtain an explicit description of the solutions  to the Laplacian coflow  on almost-abelian Lie groups   under suitable assumptions on the initial data.
\par\bigskip
Let $G$ be a seven-dimensional, simply-connected, almost-abelian Lie group  equipped with an invariant  coclosed $\G-$structure $\varphi_0$ with 4-form $\phi_0$ and let $\frak h$   be a codimension one  abelian ideal of  the Lie algebra $\frak g$ of  $G$.  By Proposition 4.5 in  \cite{Schulte}, if we choose   a  vector $e_7$ in the orthogonal complement of $\frak h$ with respect to $g_{\varphi_0}$  such that  $g_{\varphi_0}(e_7,e_7) = 1,$ the forms
\begin{equation}\label{inducedSU(3)-structure}
 \omega_0 = {e_7}   \neg \varphi_0,  \quad  \psi_0 =- e_7\neg \phi_0,
\end{equation}
define an $\SU(3)$-structure  $(\omega_0, \psi_0)$ on $\frak h$. Let $\eta=e_7\neg g_{\varphi_0}$. Then we can identify $\lie{g}^*$ with $\lie{h}^*\oplus\R\eta$ and  we have  $d \eta =0$, since $\eta$ vanishes on  the commutator $[\frak g, \frak g] \subseteq \frak h$. Moreover
$$d\alpha=\eta\wedge \vartheta(A)\alpha,
$$
for every $\alpha\in\Lambda^p\hg^*$, where $A = ad_{e_7} \vert_{\frak h}$.
In particular,  if $\phi$ is any $4$-form on $\lie{g}$, we can consider the  decomposition
\begin{equation}\label{decompAA}
\phi=\phi^{(4)}+\phi^{(3)}\wedge\eta,\quad \phi^{(i)}\in\Lambda^i\lie{h}^*,\,i=3,4.
\end{equation}
So $(\phi_0)^{(4)}=1/2\omega_0^2$ and $(\phi_0)^{(3)}=\psi_0$. Finally let us observe that  $d\phi_0=0$ if and only if $\vartheta(A)(\omega^2_0)=2\vartheta(A)(\omega_0)\wedge\omega_0=0$, which means  that $A\in\lie{sp}(\frak h, \omega_0)$,  since $\omega_0$ is  a nondegenerate $2$-form.
\begin{lemma}\label{Lem} 
Let $U$ be a real vector space of dimension $6$ endowed with a linear  $\SU(3)$-structure $(\omega, \psi)$ and $A\in\lie{sp}(\omega)$ be normal with respect to the inner product  $h$ defined by $(\omega, \psi)$.  Denote by  $J$ the complex structure induced by $\psi$ and by $S$ and $L$ the symmetric and skew-symmetric part of $A$,  respectively.  Then 
there exist  $\theta\in[0,2\pi]$ and   a basis  $\left(e_1,  e_2, e_3, J e_1,J e_2, Je_3\right)$  of $U$ such that
\begin{equation} \label{espomegapsi}
\begin{array}{l}
\omega = e^1\wedge J^*e^1 +e^2\wedge J^*e^2+e^3\wedge J^*e^3, \\[3pt]
  \Psi  = (e^1 + i  J^* e^1) \wedge (e^2 + i  J^* e^2) \wedge (e^3+ i  J^* e^3)
  \end{array}
\end{equation}
and 
\begin{equation}\label{Anorm}
S (e_i) =  s_i (\cos  (\theta)  e_i + \sin (\theta) J e_i), \quad S (J e_i) = - s_i (- \sin (\theta) e_i + \cos (\theta)  Je_i), \quad i =1,2,3,
\end{equation}
where  the real numbers $\left\{\pm s_i, i=1,2,3\right \}$ are the eigenvalues of $S$ (counted with their multiplicities), and  $J V_{s_i} = V_{- s_i}$, where   $V_{s_i}$   denotes the eigenspace of $S$ associated to the eigenvalue  $s_i$. Moreover,
\begin{enumerate}
\item if $s_j=0$,  then  $Le_j=l_j Je_j$ and $LJe_j=-l_j e_j$, for $l_j\in\R$;
\item if $s_j\neq 0$  with multiplicity $m_j$,  then    $L |_{V_{s_j} \oplus  V_{-s_j}}$ is given by the block matrix
$$
L=
\left[ \begin {array}{ccc} L' &0
\\ \noalign{\medskip} 0& L'\end {array} \right],  \quad L'\in\lie{so}(m_j), 
$$
 with respect to the basis  $(e_{i_1},\dots,e_{i_{m_j}},Je_{i_1},\dots,Je_{i_{m_j}})$ of  $ V_{s_j} \oplus  V_{-s_j}$, 
\end{enumerate}
\end{lemma}
\begin{proof}
Clearly $S$ and $L$ belong to $\lie{sp}(\omega)$ since $A$ does. Therefore we have
$$h(x,SJy)=h(Sx,Jy)=-\omega(Sx,y)=\omega(x,Sy)=-h(x,JSy),\quad x,y\in V.$$
Thus $SJ=-JS$ and, similarly, $LJ=JL$.\par
The spectrum of $S$ must be real and centrally symmetric, since $S$ is symmetric and anti-commutes with $J$. Let  $\left\{\pm s_i, i = 1,2,3 \right\}$ be the spectrum of $S$. Denote by $V_{s_i}$ the eigenspace of $S$ associated   with the eigenvalue $s_i$, and   by $m(s_i)$ its multiplicity. It is then clear that, since $[S,L]=0$ and $SJ=-JS$, $L$ preserves each eigenspace $V_{s_i}$ and $JV_{s_i}=V_{-s_i}$.\par
Now we  show that on each $J-$invariant subspace $W_{s_i}=V_{s_i}+ V_{-s_i}$, both $S$ and $L$ are given as in the statement with respect to some orthonormal basis. 
First let us consider the case when  $s_i =0$ is an eigenvalue of $S$.    Clearly   its multiplicity $m_0=m(0)$ is even and   the restriction  $L|_{V_0}$ of $L$ to  the eigenspace $V_0$  belongs to  $\lie{sp}(m_0,\R)\cap\lie{so}(m_0)=\lie{u}(m_0/2)$. Therefore we can diagonalize $L$ over $\C$ as a complex matrix finding the desired expression; indeed its eigenvalues are all imaginary numbers.\par
Now let $s_i\neq 0$ and $m(s_i)=m_i$. Then $W_{s_i}$ has real dimension $2m_i$  and there exists an orthonormal basis $\left(e_{r_1},\dots,e_{r_{m_i}}\right)$ of $V_{s_i}$ such that,  $L \vert_{W_i}$ has the following expression  with respect to the  orthonormal basis  $\left(e_{r_1},\dots,e_{r_{m_i}}, Je_{r_1},\dots,Je_{r_{m_i}}\right)$
$$
L=
\left[ \begin {array}{ccc} L_1 &L_{3}
\\ \noalign{\medskip} -L_{3}^\dagger& L_2\end {array} \right], \,  L_1,L_2\in\lie{so}(m_i),\; \, L_{3}\in\lie{gl}(m_i,\R),
$$
where by $\dagger$ we denotes the transpose. So, by $LJ=JL $ and $LS=SL$ we get $L_{3}=0$ and $L_1=L_2$.\par
Putting all together  the basis of $W_i$ we  get an orthonormal basis $ (e_1, e_2, e_3, Je_1, J e_2, J e_3)$ of $U$  but, generally,  the basis  is not an  adapted frame with respect to the linear $\SU(3)$-structure $(\omega, \psi)$. Indeed $\Psi_0=(e^1+iJ^*e^1)\wedge (e^2+iJ^*e^2)\wedge(e^3+iJ^*e^3)$ does not necessarily coincide with $\Psi$. However there exists a   complex number $z$  of modulus  1 such that $z^{-1}\Psi_0=\Psi$. If we take  a cubic root   $w$ of $z$ and we consider the linear map $Q$ defined by   $Q=\mathrm{Re}(w)\mathrm{id}+\mathrm{Im}(w)J$  we get that  $ Q^*\Psi_0=\Psi$.  The transformation $Q$  commutes with $J$ and  preserves each  vector subspace $W_{s_i}$. Moreover 
\begin{align*}
Q^*S=QSQ^{-1}=&(\mathrm{Re}(w)\mathrm{id}+\mathrm{Im}(w)J)S(\mathrm{Re}(w)\mathrm{id}-\mathrm{Im}(w)J)\\
=&S(\mathrm{Re}(w)\mathrm{id}-\mathrm{Im}(w)J)(\mathrm{Re}(w)\mathrm{id}-\mathrm{Im}(w)J)\\
=&S\left\{\left(\mathrm{Re}(w)^2-\mathrm{Im}(w)^2\right)\mathrm{id}-2\left(\mathrm{Re}(w)\mathrm{Im}(w)\right)J\right\}\\
=&\cos(\theta)S+\sin(\theta)JS,\\
Q^*L= QLQ^{-1}=&(\mathrm{Re}(w)\mathrm{id}+\mathrm{Im}(w)J)L(\mathrm{Re}(w)\mathrm{id}-\mathrm{Im}(w)J)\\
=&L(\mathrm{Re}(w)\mathrm{id}+\mathrm{Im}(w)J)(\mathrm{Re}(w)\mathrm{id}-\mathrm{Im}(w)J)\\
=&L\left(\mathrm{Re}(w)^2+\mathrm{Im}(w)^2\right)\mathrm{id}\\
=&L.
\end{align*}
Therefore the new basis  $\left(Qe_1, Q e_2, Qe_3,J Q e_1, J Q e_2,JQe_3\right)$ satisfies all the requested properties.
\end{proof}
\begin{lemma} \label{Lemma 3.2}  Let $(\frak g = \R e_7 \ltimes_{A} \frak h, \varphi_0)$ be an almost-abelian Lie algebra endowed with a coclosed $G_2$-structure $\varphi_0$. Let   $(\omega_0,    \psi_0)$ be  the induced $\SU(3)$-structure on $\frak h$ defined by \eqref{inducedSU(3)-structure}, with $\eta (e_7) \neq 0$, $\eta\vert_{\frak h} =0$ and $\| \eta \|_{g_{\varphi_0}} = 1$.
The solution $\phi_t$ of  the  Laplacian coflow on $\lie{g}$ 
\begin{equation}\label{AAcoflow}
\begin{cases}
\dot{\phi}_t=-\Delta_t\phi_t,\\
d\phi_t=0,\\
\phi_0=\star_0\varphi_0,
\end{cases}
\end{equation}
is given by
$$\phi_t=\frac{1}{2}\omega_0^2+p_t\wedge\eta,$$
where   $p_t$ is a time-dependent negative $3$-form on $\lie{h}$ solving
\begin{equation}\label{eqAA}
\begin{cases}
\dot{p}_t=-\varepsilon(p_t)^2 \,\vartheta(A)\vartheta(B_t)p_t,\\
 p_0= \psi_0,
\end{cases}
\end{equation}
where   $\varepsilon(p_t)$ is a function such that  $(\omega_0,\varepsilon(p_t)p_t)$  defines an $\SU(3)$-structure  on $\frak h$ and $B_t$  is the adjoint of $A = ad_{e_7} \vert_{\frak h}$ with respect to  the scalar product  $h_t$ induced by  the  $\SU(3)-$structure $(\omega_0,\varepsilon(p_t)p_t)$.
\end{lemma}
\begin{proof}
By Cauchy theorem the   system  of ODEs  \eqref{AAcoflow}  admits  a  unique solution. Let $\phi_t$ be the solution of \eqref{AAcoflow} and $\varepsilon_t$ be the norm $||\eta||_t$ with respect to the scalar product  $g_t$ induced by $\phi_t$. Then we can write
$$\phi_t=\frac{1}{2}\omega^2_t+\psi_t\wedge\frac{1}{\varepsilon_t}\eta,$$
where the couple $(\omega_t,\psi_t)$ defines an $\SU(3)-$structure on $\lie{h} = \mathrm{Ker}(\eta)$. To see this observe that if we define $x_t$ by $g_t(x_t,y)=\eta(y)$,  for any $y\in\lie{g}$,  then $\lie{h}=\mathrm{Ker}(\eta)=\left\{y\in\lie{g}\,|\,g_t(x_t,y)=0\right\}$. Therefore, for every $t$,   the $4$-form $\phi_t$ defines an $\SU(3)$-structure $(\omega_t, \psi_t)$  on $\frak h$.\par
With respect to the  decomposition \eqref{decompAA} we can write  $\phi_t$  as $\phi_t = \phi_t^{(4)} + \phi_t^{(3)} \wedge \eta$ with 
$$\phi^{(4)}_t=\frac{1}{2}\omega_t^2,\quad \phi^{(3)}_t=\frac{1}{\varepsilon_t}\psi_t.$$
Since the cohomology class of $\phi_t$ is fixed by the flow,  i.e.,  $\phi_t = \phi_0 + d \alpha_t$ it turns out  that 
$$\dot{\phi}_t=\dot{\phi}^{(4)}_t+ \dot{\phi}^{(3)}_t\wedge \eta = d \dot{\alpha_t} \in d\Lambda^3\lie{g}^*\subseteq \Lambda^3\lie{h}^*\wedge \R\eta.$$
Therefore $\dot{\phi}^{(4)}_t=0$, i.e.,  $\omega_t \equiv\omega_0$ and consequently 
$$
\phi_t=\frac{1}{2}\omega_0^2+\psi_t\wedge\frac{1}{\varepsilon_t}\eta.
$$
  \par
 Now define $\eta_t=\frac{1}{\varepsilon_t}\,\eta$ and denote by $\star_{g_t}$ and $\star_{h_t}$ the star  Hodge operators on $\lie{g}$ and $\lie{h}$  with respect to $g_t$ and $h_t$ respectively.  Note that
 $$
 \star_{g_t} \phi_t =   \omega_0\wedge\eta_t - \star_{h_t} \psi_t,
 $$
 since  $$
 \star_{g_t}\beta=\star_{h_t}\beta\wedge\eta_t, \quad \star_{g_t}(\beta\wedge\eta_t)=(-1)^k\star_{h_t}\beta,$$
  for every $k-$form $\beta$ on  $\lie{h}$.
 
 Then 
\begin{equation*}
\begin{array}{lcl}
\Delta_t\phi_t&=&d\star_{g_t} d\star_{g_t} \phi_t=d\star_td\left( \omega\wedge\eta_t -*_t\psi_t\right)\\
&=& - d\star_t\left(\eta\wedge \vartheta(A)*_t\psi_t\right)   =-d\star_t\left( \varepsilon_t\eta_t\wedge \vartheta(A)*_t\psi_t\right)\\
&=&-\varepsilon_td*_t\left(\vartheta(A)*_t\psi_t\right)\\
&=&\varepsilon_t\,\left(\vartheta(A)*_t\vartheta(A)*_t\psi_t\right)\wedge  \eta \\
&=&\varepsilon^2_t\,\left(\vartheta(A)*_t\vartheta(A)*_t\psi_t\right)\wedge \eta_t.
\end{array}
\end{equation*}
On the other hand we have
$$\dot{\phi}_t=\dot{\psi}_t\wedge\frac{1}{\varepsilon_t}\,\eta-\psi_t\wedge\frac{\dot{\varepsilon}_t}{\varepsilon_t^2}\,\eta=\dot{\psi}_t\wedge\eta_t-\frac{\dot{\varepsilon}_t}{\varepsilon_t}\,\psi_t\wedge\eta_t.$$
Imposing $\dot{\phi}_t = -  \Delta_t\phi_t$ we get
\begin{equation}\label{eqAA1}
\frac{d}{dt}\psi_t-\frac{d}{dt}(\varepsilon_t)\varepsilon_t^{-1}\psi_t=-\varepsilon^2_t\,\left(\vartheta(A)*_t\vartheta(A)*_t\psi_t\right).
\end{equation}\par
Consider   the $3$-form $p_t = \varepsilon_t ^{-1}\, \psi_t$. It is clear that $p_t$ is a negative $3$-form, compatible with $\omega_0$ and defining the same complex structure $J_t$ induced by $\psi_t$. Moreover it satisfies the condition $$-6\, p_t\wedge J_t^*p_t=4\, \varepsilon_t^{-2} {\omega_0}^3.$$  
 Then, by \eqref{eqAA1} we obtain
\begin{equation*}
\varepsilon_t\dot{p}_t+\dot{\varepsilon}_t p_t-\dot{\varepsilon}_tp_t = -\varepsilon_t^3\left(\vartheta(A)*_t\vartheta(A)*_t p_t\right). 
\end{equation*}
and thus   the following equation  in terms of  the $3$-form $p_t$
\begin{equation}
\dot{p}_t=-\varepsilon(p_t)^2\left(\vartheta(A)*_t\vartheta(A)*_t p_t\right),\quad p_0=  \psi_0,
\end{equation}
where  the function $\varepsilon(p_t) = \varepsilon_t  = \| \eta \|_t$   is  defined in terms of the $3$-form $p_t$  by $$6\, p_t\wedge J_t^*p_t=4\, \varepsilon(p_t)^{-2} {\omega_0}^3.$$
\par
It is easy to see that $*_t\vartheta(A)*_t$ is the $h_t-$adjoint operator of $\vartheta(A)$ on $\Lambda^3\lie{h}^*$. Indeed, if $\alpha,\beta\in\Lambda^3\lie{h}^*$,  then
$$ \langle ( (*_t \vartheta(A) *_t)\alpha,\beta \rangle_t  \,  \omega_0^3/6=-\beta\wedge  \vartheta(A)  (*_t \alpha)=\vartheta(A) (\beta)\wedge*_t(\alpha)= \langle \alpha,\vartheta(A)(\beta) \rangle_t   \,\omega_0^3/6,$$
where in the second equality we have used that $A$ is traceless and consequently   that $\vartheta(A)$ acts trivially on $6$-forms.\par 
Now let $B_t$ be the $h_t-$adjoint of $A$ on $\lie{h}$. We claim that $(*_t \vartheta(A) *_t)\alpha=\vartheta(B_t)\alpha$ for any $3$-form $\alpha$ on $\lie{h}$. To see this let $\left( e_1\dots,e_6\right)$ be an $h_t-$ortonormal  basis of $\frak h$, so\footnote{Note that we are not using the Einstein notation.}
$$(B_t)^i_{\;j}=\sum_{a,b}(A)^{a}_{\phantom{a}b}\,(h_t)_{aj}(h_t)^{bi}=(A)^j_{\phantom{j}i},\quad i,j=1,\dots,6.$$
On the other hand, for any choice of ordered triples $(i,j,k)$ and $(a,b,c)$, we get
\begin{align*}
h_t\left(\vartheta(A)e^{ijk},e^{abc}\right)&=-h_t\sum_{l,m,n}\left(A^i_{\;l}e^{ljk}+A^j_{\;m}e^{imk}+A^k_{\;n}e^{ijn},e^{abc}\right)\\
&=-\sum_{l,m,n}h_t\left(A^i_{\;i'}e^{i'j'k'}+A^j_{\;j'}e^{i'j'k'}+A^k_{\;k'}e^{i'j'k'},e^{abc}\right)\\
&=-(A^i_{\;a}+A^j_{\;b}+A^k_{\;c})
\end{align*}
and 
\begin{align*}
h_t\left(e^{ijk},\vartheta({B_t})e^{abc}\right)&=-\sum_{l,m,n}h_t\left(e^{ijk},B^a_{\;l}e^{lbc}+B^b_{\;m}e^{amc}+B^c_{\;n}e^{abn}\right)\\
&=-\sum_{l,m,n}h_t\left(e^{ijk},B^{a}_{\;a'}e^{a'b'c'}+B^{b}_{\;b'}e^{a'b'c'}+B^{c}_{\;c'}e^{a'b'c'}\right)\\
&=-(B^{a}_{\;i}+B^{b}_{\;j}+B^{c}_{\;k})\\
&=-(A^i_{\;a}+A^j_{\;b}+A^k_{\;c}),
\end{align*}
since $h_t\left(e^{ijk},e^{abc}\right)=\delta^{ia}\delta^{bj}\delta^{kc}$. Therefore
$$h_t\left(\vartheta(A)\alpha,\beta\right)=h_t\left(\alpha,\vartheta(B_t)\beta\right),\quad \alpha,\beta\in\Lambda^3\lie{h}^*,$$
as we claimed and  \eqref{eqAA} holds.
\end{proof}
\begin{thm}\label{sym}  Let $(\frak g = \R e_7 \ltimes_{A} \frak h, \varphi_0)$ be an almost-abelian Lie algebra endowed with a coclosed $G_2$-structure $\varphi_0$. Let   $(\omega_0,    \psi_0)$ be  the induced $\SU(3)$-structure on $\frak h$ defined by \eqref{inducedSU(3)-structure}, with $\eta (e_7) \neq 0$, $\eta\vert_{\frak h} =0$ and $\| \eta \|_{g_{\varphi_0}} = 1$,  and let  $J_0 = J_{\psi_0}$.
Suppose that  $A = ad_{e_7} \vert_{\frak h}$ is symmetric with respect to  the inner product $h_0 = g_0 \vert_{\frak h}$ and fix an adapted frame  $(e_1,  J_0 e_1, e_2, J_0 e_2,  e_3, J_0 e_3)$ of $(\frak h, \omega_0,   \psi_0)$ such that $\omega_0$ and $\psi_0$ are given by \eqref{espomegapsi}  and $A$  has  the normal form \eqref{Anorm}.  Furthermore  assume  that $A$ satisfies $\theta=0$.  Then the solution $p_t$ of \eqref{eqAA}
is ancient  and  it is given by
\begin{gather*}
p_t=-b_1(t)e^{246}+b_{2}(t)e^{136}+b_3(t)e^{145}+b_{4}(t)e^{235}, \quad  t \in  \left  (-\infty, \frac{1}{8 \left(s_{1}^2+s_2^2+s_3^2\right) } \right ),
\end{gather*}
where $b_i(t)=e^{-\sigma_i\epsilon(t)}$ for suitable constants $\sigma_i$ and $$\epsilon(t) = \int_0^t \frac{1}{1-8\left(s_{1}^2+s_2^2+s_3^2\right)u}du.$$
\end{thm}\begin{proof}
 Consider the following system
\begin{equation}\label{chi}
\left \{ 
\begin{array}{l}
\dot{\chi}_t=-f(t)^2\vartheta(A)\vartheta(A)\chi_t,\\[2pt]
\chi_0= \psi_0,  
\end{array} 
\right.
\end{equation}
 where  $f(t)$  is a  positive function   which will be defined   later. 
Moreover, let $$\left(f_1,f_2,f_3,f_4,f_5,f_6\right)= (e_1,  J_0 e_1, e_2, J_0 e_2,  e_3, J_0 e_3)$$ be an adapted frame of $\frak h$   such that $\omega_0$ and $\psi_0$ are given by \eqref{espomegapsi}  and $A$ has the normal form  \eqref{Anorm}. It is clear that
\begin{align*}
\vartheta(A)\vartheta(A){\psi}_0=&-(s_1+s_2+s_3)^2f^{246}\\
&+(s_1+s_2-s_3)^2f^{136}\\
&+(s_1-s_2+s_3)^2f^{145}\\
&+(-s_1+s_2+s_3)^2f^{235}.\\
\end{align*}
So 
$$\vartheta(A)\vartheta(A){\psi}_0=-\sigma_1f^{246}+\sigma_2f^{136}+\sigma_3f^{145}+\sigma_4f^{235},$$
for suitable constants $\sigma_1,\sigma_2,\sigma_3$ and $\sigma_4$.
The solution of  \eqref{chi} is then given by
\begin{equation} \label{solutionchit} \chi_t=-b_1(t)f^{246}+b_{2}(t)f^{136}+b_3(t)f^{145}+b_{4}(t)f^{235},
\end{equation}
where $b_i(t)=e^{-\sigma_i\epsilon(t)}$ for a  function $\epsilon (t)$ satisfying  $\dot{\epsilon}(t) = f(t)^2$.
In order to determine  the function $f(t)$, note that, for every $t$ where it is defined,  the $3$-form $\chi_t$ is  negative, compatible with  $\omega_0$ and it  defines a complex structure $J_t$,  given  by
\begin{equation}\label{It}
J_t=\frac{2}{\sqrt{-\nu_t}}
\left[ \begin {array}{cccccc} 0&-b_{{4}}b_{{1}}&0&0&0&0
\\ \noalign{\medskip}b_{{2}}b_{{3}}&0&0&0&0&0\\ \noalign{\medskip}0
&0&0&-b_{{3}}b_{{1}}&0&0\\ \noalign{\medskip}0&0&b_{{2}}b_{{4}
}&0&0&0\\ \noalign{\medskip}0&0&0&0&0&-b_{{2}}b_{{1}}
\\ \noalign{\medskip}0&0&0&0&b_{{3}}b_{{4}}&0\end {array} \right],\quad \nu_t=-4b_1^2b_2^2b_3^2b_4^2,
\end{equation}
with respect to the adapted frame $\left(f_1,f_2,f_3,f_4,f_5,f_6\right)$. Moreover,
$$6\,\chi_t\wedge J_t^*\chi_t=4b_1^2b_2^2b_3^2b_4^2\, {\omega_0^3}.$$
The previous condition is satisfied if  we choose $f (t)$ such  that
$$f(t)^{-2}=b_1^2b_2^2b_3^2b_4^2=e^{-2\left(\sigma_{1}+\sigma_2+\sigma_3+\sigma_4\right)\int_0^tf(u)^2 du}=e^{-8(s_1^2+s_2^2+s_3^2)\int_0^tf(u)^2du}.$$
The above  identity is  satisfied  if and only if  the function $F_t=\int_0^tf(u)^2du$ solves the following Cauchy problem
$$
\left \{
\begin{array}{l}
\dot{F}_t=e^{8 \delta F_t},\\
F_0=0,\\[2pt]
\end{array} \right.$$ 
 where $\delta=s_{1}^2+s_2^2+s_3^2$.
Integrating we get
$$t=\frac{1-e^{-8\delta F_t}}{8\delta}.
$$
Therefore $$F_t=\frac{\mathrm{ln}(1-8\delta t)}{-8\delta}$$ and consequently  $f(t)=\frac{1}{\sqrt{1-8\delta t}}.$
Finally we observe that the metric $h_t$ defined by  $(\omega_0, J_t)$ is positive definite. Moreover,   the endomorphism $H_t$, defined  by $g_0(x,H_ty)=h_t(x,y)$ for any $x,y\in\lie{h}$, has  the following matrix representation
\begin{equation}\label{hmetric}
H_t=\frac{2}{\sqrt{-\nu_t}}\left[ \begin {array}{cccccc} b_{{2}}b_{{3}}&0&0&0&0&0\\ 
\noalign{\medskip}0&b_{{1}}b_{{4}}&0&0&0&0\\
\noalign{\medskip}0&0&b_{{2}}b_{{4}}&0&0&0\\
 \noalign{\medskip}0&0&0&b_{{1}}b_{{3}}&0&0\\
 \noalign{\medskip}0&0&0&0&b_{{3}}b_{{4}}&0\\
 \noalign{\medskip}0&0&0&0&0&b_{{1}}b_{{2}}
\end {array} \right],
\end{equation}
with respect to the adapted frame $\left(f_1,f_2,f_3,f_4,f_5,f_6\right)$. Now we claim that $\chi_t$, given by \eqref{solutionchit},  with  $$\epsilon(t) =  \int_0^t f(u)^2 \, du =  \int_0^t \frac{1}{1-8\left(s_{1}^2+s_2^2+s_3^2\right)u}du,$$  is the solution of \eqref{eqAA}. To see this,  first  observe that the choice of $f (t)$ ensures that $\varepsilon(\chi_t)^2=f(t)^2$. The only thing we have to prove is that the adjoint  $C_t$ of $A$ with respect to $h_t$ is constant. It is clear that $C_t=H^{-1}_tB_0H_t$ and then the claim is equivalent to show that $[H_t,B_0]=0$.  With respect to the adapted frame $\left(f_1,f_2,f_3,f_4,f_5,f_6\right)$ the endomorphism $H_t$ is diagonal as well as $B_0=A$, and then the claim follows. Thus the solution $p_t$ of \eqref{eqAA} is given by $\chi_t$,  and in particular $B_t\equiv B_0$.
\end{proof}

\begin{rem}
The previous proof can be adapted to the case $\theta=\pi$. Indeed,  if  $\theta=\pi$ then $\vartheta(A)\vartheta(A)\psi_0$ is again a linear combination of elements  of the form $e^a\wedge e^b\wedge J_t^*e^c$ with coefficients given by a suitable choice of $\pm (s_a + s_b- s_c)$. On the other hand, when $\theta$ is different from $0$  and $\pi$,  it turns out that $\vartheta(A)\vartheta(A)\psi_0$ is a linear combination of elements $e^a\wedge e^b\wedge J_t^* e^c$ and $e^a \wedge J_t^* e^b\wedge  J_t^* e^c$. Therefore the derivative of   $J_t^*$ at $t =0$ is much more complicated than in  the other cases (see Remark \ref{remarkAnorm}).
\end{rem}
\begin{thm}\label{skew}
 Let $(\frak g = \R e_7 \ltimes_{A} \frak h, \varphi_0)$ be an almost-abelian Lie algebra endowed with a coclosed $G_2$-structure $\varphi_0$. Let   $(\omega_0,  \psi)$ be  the induced $\SU(3)$-structure on $\frak h$ defined by \eqref{inducedSU(3)-structure}, with $\eta (e_7) \neq 0$, $\eta\vert_{\frak h} =0$ and $\| \eta \|_{g_{\varphi_0}} = 1$.
Suppose that $A = ad_{e_7} \vert_{\frak h}$ is skew-symmetric with respect  to the inner product $h_0 = g_0 \vert_{\frak h}$ and define $l=l_1+l_2+l_3$, where $l_1,l_2$ and $l_3$ are as in Lemma \ref{Lem}. Then the solution $p_t$ of \eqref{eqAA}
is given by
\begin{gather*}
p_t=b(t) \psi_0,\\
\end{gather*}
where $b(t)=e^{-l^2\int_0^t \varepsilon^2_udu}$
and $\varepsilon_t$ is a positive function given by
$$\varepsilon_t=\frac{1}{\sqrt{1-2l^2t}}.$$
In particular, $p_t$ is an ancient solution, defined for every   $t$ in  $\left (-\infty,\frac{1}{2l^2} \right ).$
\end{thm}
\begin{proof}
Let $f_t$ be a positive function which will be fixed later and let us consider the following system
\begin{equation}\label{chi2}
\left \{ 
\begin{array}{l}
\dot{\chi}_t=-f^2_t\vartheta(A)\vartheta(-A)\chi_t,\\[2pt]
\chi_0= \psi_0.
\end{array}
\right.
\end{equation}
Moreover, let $\left(f_1,\dots,f_6\right)= (e_1,  J_0 e_1, e_2, J_0 e_2,  e_3, J_0 e_3)$ be an adapted  frame such that $\omega_0$ and $\psi_0$ are given by \eqref{espomegapsi}  and $A$ has the normal form  \eqref{Anorm}. It is clear that
\begin{align*}
\begin{array}{ccc}
\vartheta(A)\vartheta(A){\psi}_0=-l^2{\psi_0}.
\end{array}
\end{align*}
Therefore the solution of  \eqref{chi2} is given by
$$\chi_t=b(t) \psi_0,$$
where $b(t)=e^{-l^2\int_0^tf_u^2du}$.\par
The $3$-form $\chi_t$ is  negative, compatible with $\omega_0$ and it defines a constant complex structure $J_t\equiv J_0$. Moreover,  it
satisfies
$$6\,\chi_t\wedge I_t\chi_t=4b^2(t)\,\omega^3.$$
Now we choose $f_t$ so that
$$f^{-2}_t=b(t)^2=e^{-2l^2\int_0^tf_u^2du}.$$
To do this we solve the system 
$$ \left \{ \begin{array}{l} \dot{F}_t=e^{2l^2F_tdu},\\
 F_0=0,
 \end{array}
 \right.$$
and then we put $f_t=\sqrt{\dot{F}_t}$. Integrating by $t$  we get $$t=\frac{1-e^{2l^2 F_t}}{2l^2}.$$
Thus $$F_t=\frac{\mathrm{ln}(1-2l^2 t)}{-2l^2}
$$ 
and consequently $\varepsilon_t=\frac{1}{\sqrt{1-2l^2 t}}.$
\par
Now it is easy to show that  $\chi_t$ is a solution of   \eqref{eqAA}.  Indeed,  the choice of $f_t$ ensures that $\varepsilon(\chi_t)^2=f_t^2$ and moreover  that the metric $h_t$ induced by $\omega_0$ and $\chi_t$ is constant.  Therefore the adjoint of $A$ is constantly equal to $-A$ and, as a consequence,  the solution $p_t$ of \eqref{eqAA} is given by $\chi_t$.\par
\end{proof}
\begin{rem}\label{remarkAnorm}
It is not hard to prove that  if  $A$ is normal with respect to $h_0$,  then the solution $p_t$ of \eqref{eqAA} is given by
\begin{equation*}
p_t=-b_1(t)f^{246}+b_2(t)f^{136}+b_3(t)f^{145}+b_4(t)f^{235}+c_1(t)f^{135}-c_2(t)f^{245}-c_3(t)f^{236}-c_4(t)f^{145},
\end{equation*}
where $(f_1,\dots,f_6)=(e_1,J_0e_1, e_2, J_0 e_2, ,e_3,J_0 e_3)$ is  an adapted  frame of $\frak h$ and $A$ is given by \eqref{Anorm}. Unfortunately in this case we cannot find an explicit solution of  \eqref{eqAA}.\par
However, note  that if we write $p_t=(x_t)^*  \psi_0$,  for $[x_t]\in\GL(\lie{h})/\SL(3,\C)$, 
then $x_t$ belongs to $\GL^+(2,\R)^3$ acting on $<e_1,J_0 e_1>\oplus<e_2,J_0 e_2>\oplus<e_3,J_0 e_3>$. Therefore  $[x_t]=([x^{(1)}_t],[x^{(2)}_t],[x^{(3)}_t])\in\left((\GL^+(2,\R)/\SO(2)\right)^3$.\par
\end{rem}

\section{Solitons for the Laplacian coflow on almost-abelian Lie groups} \label{SectionSoliton}

In this section  we find sufficient  conditions  for a left-invariant  coclosed $G_2$-structure on an almost-abelian Lie  group  $G$ to be a soliton for the Laplacian  coflow. 

\smallskip

Let $\frak g$ be a Lie algebra. We recall the following 

\begin{defn} Let $\frak g$ be a  seven-dimensional Lie algebra endowed with a coclosed $G_2$-structure $\varphi_0$.  A solution $\phi_t$  to  the Laplacian coflow  \eqref{AAcoflow} on  $\frak g$  is   \emph{self-similar} if  $$\phi_t=c_t(x_t)^*\phi_0,$$
for a real-valued function $c_t$ and a $\mathrm{GL}(\lie{g})-$valued function $x_t$. 
\end{defn}

It is  well-known that a solution $\phi_t$ of \eqref{AAcoflow} is self-similar if and only if the Cauchy datum $\phi_0$ at $t=0$ is a \emph{soliton}, namely if  it satisfies
$$-\Delta_0\phi_0=-4c\phi_0+\vartheta(D)\phi_0,$$
for some real number $c$ and some derivation $D$  of the Lie algebra  $\lie{g}$ (see \cite{Lauret}).
 A soliton is said to be \emph{expanding} if $c <0$, \emph{shrinking} if $c>0$ and \emph{steady} if $c=0$.\par

Let  $K_t$  be the stabiliser of $\phi_t$ and fix a $K_t-$invariant decomposition of $\mathrm{End}(\lie{g})$. Since $\phi_t$ is stable at any time  $t$,  there exists a time-dependent endomorphism $X_t$ of  $\lie{g}$, transversal to the Lie algebra of $K_t$ (in the sense that,   for every $t$, $X_t$ takes values in an  $ad-$invariant complement of the Lie  algebra of $K_t$), such that
$$-\Delta_t\phi_t=\vartheta(X_t)\phi_t.$$ 
Therefore $\phi_0$ is a soliton  on $\frak g$ if and only if
$$X_0=c \, \mathrm{Id}+D.$$

Suppose  now that $(\frak g = \R e_7 \ltimes_A \frak h, \varphi_0)$  is  an almost-abelian Lie algebra endowed with a coclosed $G_2$-structure $\varphi_0$. 
In Lemma \ref{Lemma 3.2}  we have seen that, with no further assumptions on $A = ad_{e_7} \vert_{\frak h}$, the Laplacian coflow reads as
$$
\begin{cases}
\frac{d}{dt}\phi_t=-\varepsilon_t \vartheta(A)\vartheta(B_t)\psi_t\wedge\eta,\\
\phi_0=\star_0\varphi_0.
\end{cases}$$
We can show that  the term $-\varepsilon_t \vartheta(A)\vartheta(B_t)\psi_t\wedge\eta$ can be re-written as
\begin{equation}\label{claimXt} -\varepsilon_t\left(\vartheta(A)\vartheta(B_t)\psi_t\right)\wedge\eta=\vartheta(X_t)\phi_t,\end{equation}
for  a time-dependent endomorphism $X_t$ of $\frak g$ in the following way.

Let $(\omega_t, \psi_t)$ be  the $\SU(3)$-structure on $\frak h$ induced by $\phi_t$. By Lemma \ref{Lem} there exist  $\theta (t) \in [0, 2 \pi]$ and  an adapted frame  of $\frak h$  such that  $\eta=\varepsilon_te^7$ and the symmetric part $S(t)$ of $A$  has the normal form \eqref{Anorm}.
More precisely,  let $$a(t)=\cos(\theta(t)), \quad  b (t)=\sin(\theta(t)).$$
With respect to  the  adapted frame at time $t$, $S(t)$
has the form \eqref{Anorm}, so it is given by 
\begin{equation*}
S(t) =\left[
\begin {array}{cccc} 
S_1(t)&0&0&0\\
0&S_2(t)&0&0\\
0&0&S_3(t)&0\\
0&0&0&0\\
\end {array}
\right],
\end{equation*} 
where
\begin{equation*}
S_i (t)=\left[ \begin {array}{cc} 
a(t) s_i(t)&b(t) s_i(t)\\
\noalign{\medskip} b(t) s_i(t) &-a(t) s_i(t)\\
\end {array}
\right].
\end{equation*} 
 Define $l(t)$ to be the imaginary part of the complex trace of  the skew-symmetric part  $L(t)$  of A at time $t$ and let 
\begin{equation}\label{sigma}
\Sigma(t)=\left[
\begin {array}{cccc} 
\Sigma_1(t)&0&0&0\\
0&\Sigma_2(t)&0&0\\
0&0&\Sigma_3(t)&0\\
0&0&0&-s^2(t)\\
\end {array}
\right],\quad 
\Lambda(t)=
\left[
\begin {array}{cccc} 
0&0&0&0\\
0&0&0&0\\
0&0&0&0\\
0&0&0&-l(t)^2\\
\end {array}\right]
\end{equation}
where
\begin{gather*}
\Sigma_1 (t)=\left[ \begin {array}{cc} 
-2s_2(t)s_3(t)+4a(t)^2 s_2(t)s_3(t)&-4a(t)b(t)s_2(t)s_3(t)\\
\noalign{\medskip}-4a(t)b(t)s_2(t)s_3(t)&2s_2(t)s_3(t)-4a(t)^2s_2(t)s_3(t)\\
\end {array}
\right],\\
\Sigma_2 (t)=\left[ \begin {array}{cc} 
-2s_1(t)s_3(t)+4a(t)^2s_1(t)s_3(t)&-4a(t)b(t)s_1(t)s_3(t)\\
\noalign{\medskip}-4a(t)b(t)s_1(t)s_3(t)&2s_1(t)s_3(t)-4a(t)^2s_1(t)s_3(t)\\
\end {array}
\right],\\
\Sigma_3(t)=\left[ \begin {array}{cc} 
-2s_1(t)s_2(t)+4a^2(t)s_1(t)s_2(t)&-4a(t)b(t)s_1(t)s_2(t)\\
\noalign{\medskip}-4a(t)b(t)s_1(t)s_2(t)&2s_1(t)s_2(t)-4a(t)^2s_1(t)s_2(t)\\
\end {array}
\right],
\end{gather*}
and  $s(t)^2=s_1(t)^2+s_2(t)^2+s_3(t)^2.$
We claim that 
\begin{equation}\label{Xt}
X_t=- \varepsilon_t \left(\Sigma (t) +\Lambda(t) -[S(t),L(t)]\right).
\end{equation}
To prove this  first observe that 
 $$\begin{array}{lcl} \vartheta(A)\vartheta(B_t)\psi_t&=&\vartheta(S(t)+L(t))\vartheta(S(t)-L(t))\psi_t\\[2pt]
 &=&\left(\vartheta(S(t))\vartheta(S(t))-\vartheta(L(t))\vartheta(L(t))-\vartheta([S(t),L(t)])\right)\psi_t.
 \end{array}
 $$

Then a direct computation shows that
$$\vartheta(\Sigma(t))\phi_t=\vartheta(\Sigma(t))(\psi_t\wedge\eta)=(\vartheta(S(t))\vartheta(S(t))\psi_t)\wedge\eta.$$
On the other hand if we change the adapted frame so that $L_t$ has  the form \eqref{Anorm} we see that, as already seen in the proof of Theorem \ref{skew}, $l(t)=l_1(t)+l_2(t)+l_3(t)$ and
 $$-\vartheta(L(t))\vartheta(L(t))\psi_t=l^2(t)\psi_t,$$
which is an expression independent on the choice of the  adapted frame. 
Thus 
$$ \vartheta(\Lambda(t))\phi_t=\vartheta(\Lambda(t))(\psi_t\wedge\eta)=l(t)^2\psi_t\wedge\eta=-(\vartheta(L(t))\vartheta(L(t))\psi_t)\wedge\eta,$$
proving the claim.

\begin{thm}  \label{condiforsoliton} Let $(\frak g = \R e_7 \ltimes_A \frak h, \varphi_0)$  be an almost-abelian Lie algebra endowed with a coclosed $G_2$-structure $\varphi_0$.
The 4-form $\phi_0=\star_0\varphi_0$   is a soliton  for the Laplacian coflow if and only if $A$ satisfies 
\begin{equation}\label{AAsolitoneq}
[-\Sigma(0)+1/2[A,A^\dagger],A]=\delta A,
\end{equation}
where $A^\dagger$ denotes the transpose of $A$ with respect to the underlying metric on $\lie{h}$, $\delta =  l^2_0+s^2_0 - c$ for a constant $c \in \R$ and $\Sigma(0)$ the endomorphism \eqref{sigma}.
 If  $\phi_0$ is a soliton, then  the solution  $\phi_t$ to the Laplacian coflow  is given by
\begin{equation}\label{solution}
\phi_t=c(t)e^{f(t)D}\phi_0,\quad c(t)=\left(1-2ct\right)^{2},\quad f(t)=-\frac{1}{2c}\mathrm{ln}\left(1-2ct\right),\quad  t < \frac{1}{2c},
\end{equation}
where the derivation $D$ of $\frak g$  is  given by   $X_0-c \, \mathrm{Id}$, with  $X_0$   as in \eqref{Xt}.
\end{thm}
\begin{proof} In the light of the previous results we can write down the soliton equation for the Laplacian coflow as follows. Suppose that $\phi_0$ is a soliton, that is,  $X_0=c\mathrm{Id}+D$ for some $c\in\R$ and a derivation $D$ of $\lie{h}$. Then, by  \cite[Theorem 4.10]{Lauret}, \eqref{solution} holds. Therefore,  $A$ corresponds to a soliton if and only if there exists  $c\in\R$ such that
$D=X_0-c \, \mathrm{Id}$ is a derivation of $\frak g$. \par
This condition can be read as a system of algebraic equations for $c$ and the elements  of  the matrix associated to $A$. 
 Note that $De_7=\delta e_7,$ with $ \delta = l^2_0+s^2_0 - c$.
Hence, denoting by $\mu_A$ the Lie bracket structure defined by $A$,
$$[D,A]v=DAv-ADv=D\mu_A(e_7,v)-\mu_A(e_7,Dv)=\mu_A(De_7,v)=\delta Av, \quad v\in\lie{h}.$$
This reads as 
\begin{equation}\label{commutatorDA}  [D,A]=\delta A.
\end{equation}
Finally, writing $D=X_0-c \, \mathrm{Id}$ for $X_0$ as in \eqref{Xt} and recalling  that $[A,A^\dagger]=-2[S(0),L(0)]$, we derive \eqref{AAsolitoneq} from \eqref{commutatorDA}.
\end{proof}
We call \eqref{AAsolitoneq} the \emph{soliton equation} of the almost-abelian  Laplacian coflow.\par
Notice that we can split the soliton equation into two coupled equations, for the symmetric and skew symmetric parts of $A$, in the following way. Since the commutator of two symmetric matrices is skew-symmetric and the commutator  of a symmetric matrix and a skew-symmetric one is symmetric,  we find
\begin{equation}\label{solitoneq2}
[-\Sigma(0)+[S(0),L(0)],L(0)]=\delta S(0),\quad [-\Sigma(0)+[S(0),L(0)],S(0)]=\delta L(0).
\end{equation}
We have just proved the following result.

\begin{corol}
If a soliton $\phi_0$ of the Laplacian coflow on the  almost-abelian Lie algebra $\lie{g}= \R e_7 \ltimes_A \frak h$ is  an eigenform  of  the Laplacian then it is harmonic, namely $A\in\mathrm{su}(6)$.  
\end{corol}
\begin{proof}
Clearly,  if $\phi_0$ is an eigenform of the Laplacian,   then $D=0$, hence $X_0=c \, \mathrm{Id}$. Taking the trace of $X_0|_{\frak h}=c\mathrm{Id}|_{\frak h}$ we find $c=0$.
\end{proof}
\begin{corol}\label{normality}
 Let $(\frak g = \R e_7 \ltimes_A \frak h, \varphi_0)$  be an almost-abelian Lie algebra endowed with a coclosed $G_2$-structure $\varphi_0$. Assume that $A$ is normal with respect to the underlying metric. Then $\phi_0 = \star_0 \varphi_0$  is soliton on $\lie{g}$ if and only if $4b(1-4a^2)s_1s_2s_3=0$ and $[\Sigma(0),L(0)]=0$.  In such  a case $\delta=0$ and  hence $c\geq 0$.
\end{corol}
\begin{proof}
 By hypotheses,  equations \eqref{solitoneq2} reduce to
$$-[\Sigma(0),S(0)]=\delta L(0),\quad -[\Sigma(0),L(0)]=\delta S(0).$$
 Using  the normal form \eqref{Anorm},  a direct computation  shows that
$$[\Sigma(0),S(0)]=4b(1-4a^2)s_1s_2s_3J_0.$$
If $\delta$ was  different from zero,  then $S(0)$ would be invertible (each $s_i$ should be non-zero) and therefore $L(0)$ would be a non-zero multiple of $J_0$, contradicting Lemma \ref{Lem}. Thus $\delta=0$, that is $4b(1-4a^2)s_1s_2s_3=0$. Clearly,  if $L(0)\neq 0$,  the equation $[\Sigma(0),L(0)]=0$ is not generically satisfied.
\end{proof}
\begin{corol}
 Let $(\frak g = \R e_7 \ltimes_A \frak h, \varphi_0)$ be an almost-abelian Lie algebra endowed with a coclosed $G_2$-structure $\varphi_0$ and 
suppose that $A = ad_{e_7} \vert_{\frak h}$ is skew-symmetric with respect to the underlying metric.   Then the  solution to the Laplacian coflow obtained in Theorem  \ref{skew}  is a soliton.
\end{corol}
\begin{rem}
Differently from the Laplacian flow studied in \cite{Lauret} there exist  almost-abelian  Lie algebras  $\frak g = \R e_7 \ltimes_A \frak h$ with $A$ symmetric and admitting coclosed $\mathrm{G}_2-$structures that are no solitons for the  Laplacian coflow. Indeed,  in the light of Corollary \ref{normality} it is enough to choose a symmetric matrix $A$ and a suitable $\mathrm{G}_2-$structure for which the constant $4b(1-4a^2)s_1s_2s_3$ is non-zero. For instance  we can consider the $\mathrm{G}_2-$structure  $$  \varphi_0 = e^{127}+e^{347}+e^{567}+e^{135}-e^{146}-e^{236}-e^{245}
$$
on the Lie algebra  $\lie{g} = \R e_7 \ltimes_A \frak h$,  where $\frak h = \R < e_1, \ldots, e_6>$ and $$
A = ad_{e_7} \vert_{\frak h} =  \left[
\begin{array}{cccccc}
0&1&0&0&0&0\\
1&0&0&0&0&0\\
0&0&0&1&0&0\\
0&0&1&0&0&0\\
0&0&0&0&0&1\\
0&0&0&0&1&0\\
\end{array}
\right].$$\par
Moreover we are able to prove that the interval of existence of the corresponding solution is bounded. To this aim, and in analogy with the proof of Theorem \ref{sym}, observe that the solution to \eqref{chi} has the following expression:
$$\chi_t=-b_1(t)e^{246}+b_2(t)e^{136}+b_3(t)e^{145}+b_4(t)e^{235}.$$ 
Indeed
\begin{align*}
\vartheta(A)\vartheta(A)\chi_t=&-(3b_1(t)-2b_2(t)-2b_3(t)-2b_4(t))e^{246}\\
&+(-2b_1(t)+3b_2(t)+2b_3(t)+2b_4(t))e^{136}\\
&+(-2b_1(t)+2b_2(t)+3b_3(t)+2b_4(t))e^{145}\\
&+(-2b_1(t)+2b_2(t)+2b_3(t)+3b_4(t))e^{235}.
\end{align*}
and therefore the vector-valued function $(b_1(t),b_2(t),b_3(t),b_4(t))$ satisfies a linear ODE whose matrix is
\begin{equation*}
-f_t^2\left(\begin{matrix}
3&-2&-2&-2\\
-2&3&2&2\\
-2&2&3&2\\
-2&2&2&3\\
\end{matrix}\right).
\end{equation*}
Taking into account that this matrix is symmetric, with eigenvalues $-9f_t^2$, $-f_t^2$, $-f_t^2$, $-f_t^2$ and eigenvectors $(-1,1,1,1)$, $(1,1,0,0)$, $(1,0,1,0)$ and $(1,0,0,1)$, it follows that 
$$2b_1(t)=-e^{-9\int_0^tf_u^2du}+3e^{-\int_0^tf_u^2du},\quad 2b_2(t)=2b_3(t)=2b_4(t)=e^{-9\int_0^tf_u^2du}+e^{-\int_0^tf_u^2du}.$$
The function $F_t=\int_0^tf_u^2du$ can be fixed, as we did in Theorem \ref{sym}, by imposing
$$1=\dot{F}_t b_1^2(t)b_2^2(t)b_3^2(t)b_4^2(t)=\dot{F}_t \left(-e^{-9F_t}+3e^{-F_t}\right)^2\left(e^{-9F_t}+e^{-F_t}\right)^6\frac{1}{32}.$$
This guarantees that $\chi_t$ actually solves \eqref{eqAA} (note also that $A$ is symmetric for any time).\par 
Notice that the previous equation, after integration, ensures that, since $F_t\geq 0$ if and only if $t\geq 0$, the solution extinguishes in finite time.  With an analogous argument we see that $F_t$ cannot exist for any negative time. To be more precise let $I$ be the maximal interval of existence of $F$, then
$$32t=\int_0^{F_t}(-e^{-9x}+3e^{-x})^2(e^{-9x}+e^{-x})^6 dx,\quad t\in I.$$
We immediately see that $\mathrm{sup}_I<+\infty$. On the other hand, if $\mathrm{inf}_I=-\infty$ then $F_t$ should be unbounded near $-\infty$: indeed when $M<F_t<0$ it turns out that
$$32 t=\int_0^{F_t}(-e^{-9x}+3e^{-x})^2(e^{-9x}+e^{-x})^6 dx>\int_0^{M}(-e^{-9x}+3e^{-x})^2(e^{-9x}+e^{-x})^6 dx.$$
 Therefore it would exist a sufficiently large negative time $t$ such that $0=-e^{-9F_t}+3e^{-F_t}=2b_1(t)$. Clearly this cannot happen because $\chi_t$ must be a stable and negative form. By these considerations we also deduce that the only negative singular time $\tau$ for the monotone function $F_t$ satisfies $b_1(\tau)=0$, that is $F_\tau=-1/8\mathrm{ln}(3)$. 
\end{rem}

We will now construct an explicit example of soliton on a nilpotent almost-abelian Lie group.

\begin{ex} Let $\frak g$ be  the nilpotent almost-abelian Lie algebra with structure equations
$$
\begin{array}{l} d e^1 = e^{27}, \\[2pt]
 d e^j =0,  \, \, j = 2,4, 6, 7,\\[2pt]
 d e^3 = e^{47},\\[2pt]
 d e^5 = e^{67}.
 \end{array} 
 $$
Then in this case we have $\frak h = \R < e_1, \ldots, e_6 >$ and 
$$
A = ad_{e_7} \vert_{\frak h} = \left (  \begin{array}{cccccc}
 0&1&0&0&0&0\\
 0&0&0&0&0&0\\
 0&0&0&1&0&0\\
 0&0&0&0&0&0\\
 0&0&0&0&0&1\\
 0&0&0&0&0&0 \end{array}
  \right).
  $$
Consider the  $G_2$-structure $\varphi_0 = e^{127} + e^{347} + e^{567} + e^{135} - e^{146} - e^{236} - e^{245}.$   The $4$-form 
$$
\phi_0  =  \star_{\varphi_0}  \, \varphi_0 =  e^{1234} + e^{3456} + e^{1256} - e^{2467} + e^{1367} + e^{1457} - e^{2357}
$$
is closed and thus $\varphi$ defines a coclosed $G_2$-structure.
The basis $(e_1, \dots , e_7)$ is orthonormal with respect to $g_{\varphi_0}$ and one can check that $A$ is not normal.  We will  apply Theorem \ref{condiforsoliton}  to show that 
$\phi_0$ is a soliton for the Laplacian coflow.  First observe that $S(0)$ and $L(0)$, on $\lie{h}$, restrict to 
$$
 \left (  \begin{array}{cccccc}
 0&1/2&0&0&0&0\\
 1/2&0&0&0&0&0\\
 0&0&0&1/2&0&0\\
 0&0&1/2&0&0&0\\
 0&0&0&0&0&1/2\\
 0&0&0&0&1/2&0 \end{array}
  \right),\quad  
  \left (  \begin{array}{cccccc}
 0&1/2&0&0&0&0\\
 -1/2&0&0&0&0&0\\
 0&0&0&1/2&0&0\\
 0&0&-1/2&0&0&0\\
 0&0&0&0&0&1/2\\
 0&0&0&0&-1/2&0 \end{array}
  \right).
  $$
So $S(0)$ is in normal form \eqref{Anorm} and therefore the  matrix  $\Sigma(0)$, restricted to $\lie{h}$, turns out to be
$$
\left (  \begin{array}{cccccc} 
  -1/2& 0& 0 &0& 0& 0\\
 0& 1/2 &0 &0 &0 &0\\
 0 &0 &-1/2 &0& 0& 0\\
 0 &0 &0 &1/2 &0 &0\\
 0 &0 & 0 & 0 & -1/2 & 0\\
 0 &0 & 0  & 0 & 0 & 1/2
 \end{array}
 \right ).
 $$
A direct computation then shows that $[S(0),L(0)]=\Sigma(0)$ on $\lie{h}$, which leads to $[-\Sigma(0)+[S(0),L(0)],A]=0$, so $A$ solves the soliton equation for $\delta=0$.
In particular we have
 $ s_1(0)=s_2(0)=s_3(0)=1/2$  and  $l(0)=l_1(0)+l_2(0)+l_3(0) = 3/2$. Thus  $$
 s^2(0)= 3/4, \quad l^2(0)= 9/4$$
 and   $c=3$.
Then the associated derivation $D$ is given by 
 $D=X-3 \, {\rm{Id}}$, with 
$$
X = \left( \begin{array}{ccccccc} 0 & 0 & 0 & 0 &  0 & 0&0\\
 0 & 0 & 0 & 0  & 0 & 0 & 0\\
 0 & 0  & 0 &  0  & 0  & 0  & 0\\
 0 &  0 &  0 & 0 &  0  & 0  & 0\\
 0 & 0 & 0 & 0  & 0 &  0 &  0\\
 0 & 0 & 0 & 0 & 0  &  0 &  0\\
 0 & 0  & 0 &  0 & 0 & 0  & 3
 \end{array} \right),
 $$
and the existence interval is   $(-\infty,1/6)$. Note that  $\phi_0$ is not an eigenform of the Laplacian since 
$$\vartheta(X)\phi_0=-3 (- e^{2467} + e^{1367} + e^{1457} - e^{2357}).$$ 
\end{ex}

\section*{Acknowledgements}  {The authors are supported by  the project PRIN \lq \lq variet\`a  reali e complesse:
geometria, topologia e analisi armonica" and by GNSAGA. of I.N.d.A.M.}
The authors also tank the anonymous referees for all the useful improvements they suggested.


\begin{thebibliography}{69}

\bibitem{BFF} L. Bagaglini,  M. Fern\'andez, A. Fino,  Coclosed $G_2$-structures induced on nilsolitons,  {\em  arXiv:1611.05264},  to appear in  {\it Forum Math.}

\bibitem{BFF2}  L. Bagaglini,  M. Fern\'andez, A. Fino, Laplacian coflow on  the  7-dimensional
 Heisenberg group, {\em arXiv:1704.00295}.
 
\bibitem{Bag17} L. Bagaglini, A flow of isometric $G_2$-structures. Short-time existence,
{\em arXiv:1709.06356}.

\bibitem{Bryant2} R. L. Bryant, Some remarks on $G_2$-structures, Proceedings of  G\"okova Geometry-Topology
Conference 2005, 75-109, G\"okova Geometry/Topology Conference (GGT), G\"okova 2006.

\bibitem{FFM} M. Fern\'andez, A. Fino, and V. Manero,  Laplacian 
flow of closed $G_2$-structures inducing nilsolitons, {\em J. Geom. Anal.}  {\bf 26(3)}  (2016), 1808-1837.

\bibitem{CLSS} 
V. Cort\'es, V. ,  T. Leistner, L. Sh\"afer,   and F: Schulte-Hengesbach, {Half-flat Structures and Special Holonomy}, 
{\em Proc. London Math. Soc.\/} {\bf 102} (2011), 113-158.

\bibitem{FernandezGray}  M. Fern\'andez and  A. Gray, Riemannian manifolds with structure group $G_2$,  {\em Ann. Mat.
Pura Appl. (4)} {\bf 132} (1982), no. 1, 19 -45.

\bibitem{FR} A. Fino, A. Raffero,  Closed warped $G_2$-structures evolving under the Laplacian  flow,  {\em arXiv:1708.00222}.

\bibitem{Freibert1} {M. Freibert},  Cocalibrated structures on Lie algebras with a codimension one abelian ideal,  {\em Ann. Global Anal. Geom.}  {\bf 42} (2012) 537-563.

\bibitem{Freibert2}  {M. Freibert} Calibrated and parallel structures on almost-abelian Lie algebras,  {\em arXiv: 1307.2542}.

\bibitem{Gri13}{S. Grigorian}, 
Short-time behaviour of a modified Laplacian coflow of $\mathrm{G}_2-$structures,
{\em Advances in Mathematics} {\bf 248} (2013),  378-415.

\bibitem{Gri15}{S. Grigorian}, 
Modified Laplacian coflow of $\mathrm{G}_2-$structures on manifolds with symmetry,
{\em Differential Geom. Appl.} {\bf 46} (2016),  38-78.


\bibitem{HarveyLawson} R. Harvey, H. B. Lawson, Calibrated geometries, {\em Acta Math.}  {\bf 148} (1982), 47-157.

\bibitem{Hitchin} 
N.~Hitchin, The geometry of three-forms in six and seven dimensions, 
{\em J.  Diff. Geom.\/} {\bf 55} (2000), 547-576.


\bibitem{KMT} S.  Karigiannis, B. McKay, M-P. Tsui,  Soliton solutions for the Laplacian coflow of some $G_2$-structures with symmetry, {\em Diff. Geom. Appl.}  {\bf 30} (2012), 318-333.


\bibitem{Lauret} 
{J. Lauret}, {Geometric Flows and their solitons on homogeneous spaces}, 
{\em  arXiv:1507.08163\/},  to appear in  {\it Rend. Sem. Mat. Torino}.

\bibitem{Lauret2}  J. Lauret, Laplacian flow of homogeneous $G_2$-structures and its solitons. {\em Proc. Lond. Math. Soc. (3)}
{\bf 114(3)} (2017),  527-560.

\bibitem{Lauret3}  J. Lauret,  Laplacian solitons: questions and homogeneous examples,  {\em arXiv:1703.01853}, to appear in
{\em Diff. Geom. Appl.}

\bibitem{LW1} J. D. Lotay, Y. Wei, Laplacian 
flow for closed $G_2$ structures: Shi-type estimates, uniqueness and
compactness,  {\em Geom. Funct. Anal.}  {\bf 27(1)} (2017), 165-233.

\bibitem{LW2} J. D. Lotay, Y. Wei,  Stability of torsion-free $G_2$ structures along the Laplacian flow, {\em arXiv:1504.07771}, to appear in {\em J. Differential Geom.}

\bibitem{LW3}  J. D. Lotay, Y. Wei,  Laplacian flow for closed $G_2$ structures: real analyticity, {\em arXiv:1601.04258}, to
appear in {\em Comm. Anal. Geom.}

\bibitem{Nicolini} M. Nicolini, Laplacian solitons on nilpotent Lie groups, {\em arXiv:1608.08599}.

\bibitem{Schulte} {F. Schulte-Hengesbach},  Half-flat structures on Lie groups, PhD Thesis
(2010), Hamburg, available at http://www.math.uni-hamburg.de/home/schulte-
hengesbach/diss.pdf.

\bibitem{WW12} {H. Weiss, F. Witt},
{A heat flow for special metrics},
{\em Adv. Math.}
{\bf 231(6)},
(2012), 
{3288-3322}.

\end{thebibliography}
\end{document}